\documentclass[12pt,lenq]{amsart}
\usepackage{amsmath}
\usepackage{amscd}
\usepackage{amssymb}
\usepackage{amsbsy}
\usepackage{amsfonts}
\usepackage{latexsym}
\usepackage{graphics}
\usepackage{graphicx}
\usepackage{amsmath,amscd,latexsym}
\usepackage{multirow}
\usepackage{array}
\usepackage{paralist}
\usepackage{titletoc}

\theoremstyle{plain}
\newtheorem{theorem}{Theorem}[section]
\newtheorem{proposition}[theorem]{Proposition}
\newtheorem{lemma}[theorem]{Lemma}
\newtheorem{corollary}[theorem]{Corollary}
\newtheorem{remark}[theorem]{Remark}
\newtheorem{definition}[theorem]{Definition}

\newtheorem{main theorem}[theorem]{Main Theorem}

\newcommand{\ZZ}{\mathbb{Z}}
\newcommand{\QQ}{\mathbb{Q}}
\newcommand{\RR}{\mathbb{R}}

\newcommand{\HH}{\mathbb{H}}
\newcommand{\QQQ}{\hat{\mathbb{Q}}}
\newcommand{\RRR}{\hat{\mathbb{R}}}

\newcommand{\Conway}{\mbox{\boldmath$S$}^{2}}
\newcommand{\Conways}
{(\mbox{\boldmath$S$}^{2},\mbox{\boldmath$P$})}
\newcommand{\PP}{\mbox{\boldmath$P$}}
\newcommand{\PConway}{\mbox{\boldmath$S$}}

\newcommand{\DD}{\mathcal{D}}
\newcommand{\RGPP}[1]{\hat\Gamma_{#1}}
\newcommand{\RGP}[1]{\Gamma_{#1}}

\newcommand{\svert}{\,|\,}

\newcommand{\llangle}{\langle\langle}
\newcommand{\rrangle}{\rangle\rangle}

\newcommand{\lp}{(\hskip -0.07cm (}
\newcommand{\rp}{)\hskip -0.07cm )}

\begin{document}

\title[A new proof for small cancellation conditions of 2-bridge link groups]
{A new proof for small cancellation conditions of 2-bridge link groups}

\author[Daewa Kim]{Daewa Kim}
\address{Department of Mathematics\\
Pusan National University \\
San-30 Jangjeon-Dong, Geumjung-Gu, Pusan, 609-735, Korea}
\email{happybug21@hanmail.net}

\author[Donghi Lee]{Donghi Lee}
\address{Department of Mathematics\\
Pusan National University \\
San-30 Jangjeon-Dong, Geumjung-Gu, Pusan, 609-735, Korea}
\email{donghi@pusan.ac.kr}

\subjclass[2010]{Primary 57M25, 20F06 \\
\indent {This work was supported by a 2-Year Research Grant of Pusan National University.}}

\begin{abstract}
In this paper, we give a simple proof
for the small cancellation conditions of
the upper presentations of 2-bridge link groups,
which holds the key to the proof of the main result of [1].
We also give an alternative proof of the main result of
[1] using transfinite induction.
\end{abstract}

\maketitle

\section{Introduction}
\label{sec:introduction}

In \cite{lee_sakuma}, the second author and M. Sakuma
gave a complete characterization of
those essential simple loops in a $2$-bridge sphere of a $2$-bridge link which are null-homotopic in the link complement,
and by using the result, they described all upper-meridian-pair-preserving epimorphisms between $2$-bridge link groups.
The main purpose of this paper is to give a simple proof
for the small cancellation conditions of
the upper presentations of 2-bridge link groups,
which holds the key to the proof of the main result of [1].
We also give an alternative proof of the main result of
[1] using transfinite induction.
It is well-known that $2$-bridge links, $K(r)$, are parametrized by extended rational numbers, $r$,
and that by Shubert's classification of $2$-bridge links ~\cite{Schubert}, it suffices to consider $K(r)$ for $r=\infty$ or $0<r\le 1$.
Here if $r=\infty$ or $r=1$, then $K(r)$ becomes a trivial $2$-component link or
a trivial knot, respectively. Since these trivial cases are easy to treat for our purpose (see \cite[Section ~7]{lee_sakuma}),
we may assume $0<r<1$. Then such a rational number $r$ is uniquely expressed in the following continued fraction expansion:
\begin{center}
\begin{picture}(230,70)
\put(0,48){$\displaystyle{
r=
\cfrac{1}{m_1+
\cfrac{1}{ \raisebox{-5pt}[0pt][0pt]{$m_2 \, + \, $}
\raisebox{-10pt}[0pt][0pt]{$\, \ddots \ $}
\raisebox{-12pt}[0pt][0pt]{$+ \, \cfrac{1}{m_k}$}
}} \
=:[m_1,m_2, \dots,m_k],}$}
\end{picture}
\end{center}
where $k \ge 1$, $(m_1, \dots, m_k) \in (\mathbb{Z}_+)^k$, and $m_k \ge 2$.

In \cite{lee_sakuma}, the proofs of key lemmas and propositions such as Lemma ~7.3 and Propositions ~4.3, 4.4 and 4.5
proceed by induction on $k$, the length of the continued fraction expansion of $r$,
where a rational number $\tilde{r}$ defined by $\tilde{r}=[m_2-1, \dots,m_k]$ if $m_2 \ge 2$ and $\tilde{r}=[m_3, \dots,m_k]$ if $m_2=1$
plays an important role as a predecessor of $r=[m_1,m_2, \dots,m_k]$ (see \cite[Proposition ~4.4]{lee_sakuma}).

However, in this paper, we define a well-ordering $\preceq$ on the set of rational numbers
greater than $0$ and less than $1$ (see Definition ~\ref{def:well-ordering}),
and then prove key lemmas and propositions such as Lemmas ~\ref{lem:properties} and \ref{lem:connection},
and Propositions ~\ref{prop:CS-sequence} and \ref{prop:decomposition}
using transfinite induction with respect to $\preceq$,
where a rational number $\tilde{r}$ defined by $\tilde{r}=[m_1-1, \dots,m_k]$
if $m_1 \ge 2$ or $\tilde{r}=[m_2+1, \dots,m_k]$ if $m_1=1$ plays a role
as a predecessor of $r=[m_1,m_2, \dots,m_k]$ (see Lemma ~\ref{lem:inductive_step}).
Note that having a smaller gap between $r$ and $\tilde{r}$ than in \cite{lee_sakuma}
makes the proof less complicated.

This paper is organized as follows.
In Section ~\ref{sec:main_statement},
we describe the main statement that we are going to re-prove in the present paper.
In Section ~\ref{sec:upper_presentation},
we recall the upper presentation of a 2-bridge link group.
In Section ~\ref{sec:new_proof_small_cancellation},
we re-prove key lemmas and propositions with some modification, if necessary,
to the original statements established in \cite{lee_sakuma}.
Finally, Section ~\ref{sec:new_proof_theorem} is devoted to
a new proof of the main theorem.

\section{Main statement}
\label{sec:main_statement}

For a rational number $r \in \QQQ:=\QQ\cup\{\infty\}$,
let $K(r)$ be the $2$-bridge link of slope $r$,
which is defined as the sum
$(S^3,K(r))=(B^3,t(\infty))\cup (B^3,t(r))$
of rational tangles of slope $\infty$ and $r$.
The common boundary $\partial (B^3,t(\infty))= \partial (B^3,t(r))$
of the rational tangles is identified
with the {\it Conway sphere} $\Conways:=(\RR^2,\ZZ^2)/H$,
where $H$ is the group of isometries
of the Euclidean plane $\RR^2$
generated by the $\pi$-rotations around
the points in the lattice $\ZZ^2$.
Let $\PConway$ be the $4$-punctured sphere $\Conway-\PP$
in the link complement $S^3-K(r)$.
Any essential simple loop in $\PConway$, up to isotopy, is obtained as
the image of a line of slope $s\in\QQQ$ in $\RR^2-\ZZ^2$
by the covering projection onto $\PConway$.
The (unoriented) essential simple loop in $\PConway$ so obtained
is denoted by $\alpha_s$.
We also denote by $\alpha_s$ the conjugacy class of
an element of $\pi_1(\PConway)$
represented by (a suitably oriented) $\alpha_s$.
Then the {\it link group} $G(K(r)):=\pi_1(S^3-K(r))$
is identified with
$\pi_1(\PConway)/ \llangle\alpha_{\infty},\alpha_r\rrangle$,
where $\llangle\cdot \rrangle$ denotes the normal closure.

Let $\DD$ be the {\it Farey tessellation},
whose ideal vertex set is identified with $\QQQ$.
For each $r\in \QQQ$,
let $\RGP{r}$ be the group of automorphisms of
$\DD$ generated by reflections in the edges of $\DD$
with an endpoint $r$, and
let $\RGPP{r}$ be the group generated by $\RGP{r}$ and $\RGP{\infty}$.
Then the region, $R$, bounded by a pair of
Farey edges with an endpoint $\infty$
and a pair of Farey edges with an endpoint $r$
forms a fundamental domain of the action of $\RGPP{r}$ on $\HH^2$
(see Figure ~\ref{fig.fd}).
Let $I_1$ and $I_2$ be the closed intervals in $\RRR$
obtained as the intersection with $\RRR$ of the closure of $R$.
Suppose that $r$ is a rational number with $0<r<1$.
(We may always assume this except when we treat the
trivial knot and the trivial $2$-component link.)
Write $r=[m_1,m_2, \dots,m_k]$,
where $k \ge 1$, $(m_1, \dots, m_k) \in (\mathbb{Z}_+)^k$, and $m_k \ge 2$.
Then the above intervals are given by
$I_1=[0,r_1]$ and $I_2=[r_2,1]$,
where
\begin{align*}
r_1 &=
\begin{cases}
[m_1, m_2, \dots, m_{k-1}] & \mbox{if $k$ is odd,}\\
[m_1, m_2, \dots, m_{k-1}, m_k-1] & \mbox{if $k$ is even,}
\end{cases}\\
r_2 &=
\begin{cases}
[m_1, m_2, \dots, m_{k-1}, m_k-1] & \mbox{if $k$ is odd,}\\
[m_1, m_2, \dots, m_{k-1}] & \mbox{if $k$ is even.}
\end{cases}
\end{align*}

\begin{figure}[h]
\begin{center}
\includegraphics{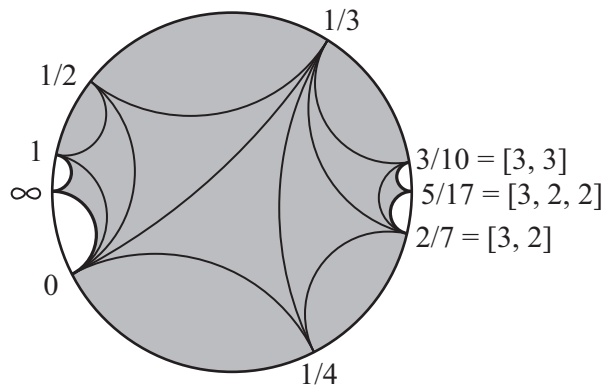}
\end{center}
\caption{\label{fig.fd}
A fundamental domain of $\hat\Gamma_r$ in the
Farey tessellation (the shaded domain) for $r=5/17=[3,2,2]$.}
\end{figure}

We recall the following fact
(\cite[Proposition ~4.6 and Corollary ~4.7]{Ohtsuki-Riley-Sakuma}
and \cite[Lemma ~7.1]{lee_sakuma})
which describes the role of $\RGPP{r}$ in the study of
$2$-bridge link groups.

\begin{proposition}
\label{prop:orbit}
{\rm (1)}
If two elements $s$ and $s'$ of $\QQQ$ belong to the same orbit $\RGPP{r}$-orbit,
then the unoriented loops $\alpha_s$ and $\alpha_{s'}$ are homotopic in $S^3-K(r)$.

{\rm (2)}
For any $s\in\QQQ$,
there is a unique rational number
$s_0\in I_1\cup I_2\cup \{\infty, r\}$
such that $s$ is contained in the  $\RGPP{r}$-orbit of $s_0$.
In particular, $\alpha_s$ is homotopic to $\alpha_{s_0}$ in
$S^3-K(r)$.
Thus if $s_0\in\{\infty, r\}$, then $\alpha_s$ is null-homotopic
in $S^3-K(r)$.
\end{proposition}

The following theorem proved in \cite{lee_sakuma} and to be re-proved
in Section ~\ref{sec:new_proof_theorem} of the present paper
shows that the converse to Proposition ~\ref{prop:orbit}(2) also holds.

\begin{theorem}
\label{thm:main_result}
The loop $\alpha_s$ is null-homotopic in $S^3 - K(r)$
if and only if $s$ belongs to the $\RGPP{r}$-orbit of $\infty$ or $r$.
In other words, if $s\in I_1\cup I_2$, then
$\alpha_s$ is not null-homotopic in $S^3-K(r)$.
\end{theorem}

\section{Upper presentations of 2-bridge link groups}
\label{sec:upper_presentation}

Throughout this paper,
the set $\{a,b\}$ denotes the standard meridian-generator
of the rank $2$ free group $\pi_1(B^3-t(\infty))$,
which is specified as in \cite[Section ~3]{lee_sakuma}.
For a positive rational number $q/p$, where $p$ and $q$ are relatively prime positive integers,
let $u_r$ be the word in $\{a,b\}$ obtained as follows.
(For a geometric description, see \cite[Remark ~1]{lee_sakuma}.)
Set $\epsilon_i = (-1)^{\lfloor iq/p \rfloor}$,
where $\lfloor x \rfloor$ is the greatest integer not exceeding $x$.
\begin{enumerate}[\indent \rm (1)]
\item If $p$ is odd, then
\[u_{q/p}=a\hat{u}_{q/p}b^{(-1)^q}\hat{u}_{q/p}^{-1},\]
where
$\hat{u}_{q/p} = b^{\epsilon_1} a^{\epsilon_2} \cdots b^{\epsilon_{p-2}} a^{\epsilon_{p-1}}$.
\item If $p$ is even, then
\[u_{q/p}=a\hat{u}_{q/p}a^{-1}\hat{u}_{q/p}^{-1},\]
where
$\hat{u}_{q/p} = b^{\epsilon_1} a^{\epsilon_2} \cdots a^{\epsilon_{p-2}} b^{\epsilon_{p-1}}$.
\end{enumerate}
Then $u_r\in F(a,b)\cong\pi_1(B^3-t(\infty))$ is represented by the simple loop $\alpha_r$,
and the link group $G(K(r))$ with $r>0$ has the following presentation,
called the {\it upper presentation}:
\[
\begin{aligned}
G(K(r))&=\pi_1(S^3-K(r))\cong\pi_1(B^3-t(\infty))/\llangle \alpha_r\rrangle \\
&\cong F(a, b)/ \llangle u_r \rrangle
\cong \langle a, b \, | \, u_r \rangle.
\end{aligned}
\]

We recall the definition of the sequence $S(r)$ and the cyclic sequence $CS(r)$
of slope $r$ defined in \cite{lee_sakuma},
both of which are read from the single relator $u_r$ of the upper presentation of $G(K(r))$.
We first fix some definitions and notation.
Let $X$ be a set.
By a {\it word} in $X$, we mean a finite sequence
$x_1^{\epsilon_1}x_2^{\epsilon_2}\cdots x_n^{\epsilon_n}$
where $x_i\in X$ and $\epsilon_i=\pm1$.
Here we call $x_i^{\epsilon_i}$ the {\it $i$-th letter} of the word.
For two words $u, v$ in $X$, by
$u \equiv v$ we denote the {\it visual equality} of $u$ and
$v$, meaning that if $u=x_1^{\epsilon_1} \cdots x_n^{\epsilon_n}$
and $v=y_1^{\delta_1} \cdots y_m^{\delta_m}$
($x_i, y_j \in X$; $\epsilon_i, \delta_j=\pm 1$),
then $n=m$ and $x_i=y_i$ and $\epsilon_i=\delta_i$ for each
$i=1, \dots, n$.
The length of a word $v$ is denoted by $|v|$.
A word $v$ in $X$ is said to be {\it reduced}
if $v$ does not contain $xx^{-1}$ or $x^{-1}x$ for any $x \in X$.
A word is said to be {\it cyclically reduced}
if all its cyclic permutations are reduced.
A {\it cyclic word} is defined to be the set of
all cyclic permutations of a cyclically reduced word.
By $(v)$ we denote the cyclic word associated with a
cyclically reduced word $v$.
Also by $(u) \equiv (v)$ we mean the {\it visual equality} of two cyclic words
$(u)$ and $(v)$. In fact, $(u) \equiv (v)$ if and only if $v$ is visually a cyclic shift
of $u$.

\begin{definition}
\label{def:alternating}
{\rm (1) Let $v$ be a nonempty reduced word in
$\{a,b\}$. Decompose $v$ into
\[
v \equiv v_1 v_2 \cdots v_t,
\]
where, for each $i=1, \dots, t-1$, all letters in $v_i$ have positive (resp., negative) exponents,
and all letters in $v_{i+1}$ have negative (resp., positive) exponents.
Then the sequence of positive integers
$S(v):=(|v_1|, |v_2|, \dots, |v_t|)$ is called the {\it $S$-sequence of $v$}.

(2) Let $(v)$ be a nonempty cyclic word in
$\{a, b\}$. Decompose $(v)$ into
\[
(v) \equiv (v_1 v_2 \cdots v_t),
\]
where all letters in $v_i$ have positive (resp., negative) exponents,
and all letters in $v_{i+1}$ have negative (resp., positive) exponents (taking
subindices modulo $t$). Then the {\it cyclic} sequence of positive integers
$CS(v):=\lp |v_1|, |v_2|, \dots, |v_t| \rp$ is called
the {\it cyclic $S$-sequence of $(v)$}.
Here, the double parentheses denote that the sequence is considered modulo
cyclic permutations.

(3) A nonempty reduced word $v$ in $\{a,b\}$ is said to be {\it alternating}
if $a^{\pm 1}$ and $b^{\pm 1}$ appear in $v$ alternately,
i.e., neither $a^{\pm2}$ nor $b^{\pm2}$ appears in $v$.
A cyclic word $(v)$ is said to be {\it alternating}
if all cyclic permutations of $v$ are alternating.
In the latter case, we also say that $v$ is {\it cyclically alternating}.
}
\end{definition}

\begin{definition}
\label{def4.1(3)}
{\rm
For a rational number $r$ with $0<r\le 1$,
let $G(K(r))=\langle a, b \, | \, u_r \rangle$ be the upper presentation.
Then the symbol $S(r)$ (resp., $CS(r)$) denotes the $S$-sequence $S(u_r)$ of $u_r$
(resp., cyclic $S$-sequence $CS(u_r)$ of $(u_r)$), which is called
the {\it S-sequence of slope $r$}
(resp., the {\it cyclic S-sequence of slope $r$}).}
\end{definition}

The following is cited from \cite{lee_sakuma}. Since its proof in \cite{lee_sakuma}
is irrelevant to the modification to be performed in the present paper,
we adopt the proof as it is.

\begin{lemma}{\cite[Proposition ~4.2]{lee_sakuma}}
\label{lem:half-rotation}
For the positive rational number $r=q/p$,
the sequence $S(r)$ has length $2q$,
and it represents the cyclic sequence $CS(r)$.
Moreover the cyclic sequence $CS(r)$ is invariant by
the half-rotation; that is, if $s_j(r)$ denotes the $j$-th term of $S(r)$
$(1\le j\le 2q)$, then $s_j(r)=s_{q+j}(r)$ for every integer $j$ $(1\le j\le q)$.
\end{lemma}

\section{New proof for small cancellation conditions of 2-bridge link groups}
\label{sec:new_proof_small_cancellation}

In this section, we give new proofs to several lemmas and propositions
with some modification, if necessary, to the original statements established in \cite[Section ~4]{lee_sakuma}.
These will play crucial roles in the new proof of Theorem ~\ref{thm:main_result}.

In the remainder of this paper unless specified otherwise,
we suppose that $r$ is a rational number
with $0<r \le 1$, and write $r$ as a continued fraction:
\[
r=[m_1,m_2, \dots,m_k],
\]
where $k \ge 1$, $(m_1, \dots, m_k) \in (\mathbb{Z}_+)^k$ and
$m_k \ge 2$ unless $k=1$.

\begin{lemma}
\label{lem:inductive_step}
For a rational number $r=[m_1, m_2, \dots, m_k]$ with $0<r<1$,
let $\tilde{r}$ be a rational number defined as
\[
{\tilde r}=
\begin{cases}
[m_1-1, m_2, m_3, \dots, m_k] & \text{if \ $m_1\ge 2$};\\
[m_2+1, m_3, m_4, \dots, m_k] & \text{if \ $m_1=1$}.
\end{cases}
\]
Then we have
\[
r=
\begin{cases}
\tilde {r}/(1+ \tilde{r}) & \text{if \ $m_1\ge 2$};\\
1-\tilde{r} & \text{if \ $m_1=1$}.
\end{cases}
\]
\end{lemma}

\begin{proof}
If $m_1\ge 2$, then letting $a:=1/\tilde{r}=m_1-1+[m_2, \dots, m_k]$
we have
\[
r=[m_1, m_2, \dots, m_k]=1/(1+a)=1/(1+1/\tilde{r})=\tilde{r}/(1+\tilde{r}),
\]
as required.

On the other hand, if $m_1=1$, then letting $b:=1/\tilde{r}-1=m_2+[m_3, \dots, m_k]$
we have
\[
r=[m_1, m_2, \dots, m_k]=1/(1+1/b)=1/(1+\tilde{r}/(1-\tilde{r}))
=1-\tilde{r},
\]
as required.
\end{proof}

\begin{proposition}
\label{prop:CS-sequence}
For a rational number $r=[m_1, m_2, \dots, m_k]$ with $0<r<1$,
let $\tilde{r}$ be a rational number defined as in
Lemma ~\ref{lem:inductive_step}.
Put $CS({\tilde r})=\lp a_1, a_2, \dots, a_t, a_1, a_2, \dots, a_t \rp$.
Then the following hold.
\begin{enumerate}[\indent \rm (1)]
\item If $m_1\ge 2$, then
\[CS(r)=\lp a_1+1, a_2,+1, \dots, a_t+1, a_1+1, a_2+1, \dots, a_t+1 \rp.\]

\item If $m_1=1$, then every $a_i$ is at least $2$, and either
\[CS(r)=\lp 2, b_1 \langle 1 \rangle, 2, b_2 \langle 1 \rangle, \dots, 2, b_t \langle 1 \rangle,
2, b_1 \langle 1 \rangle, 2, b_2 \langle 1 \rangle, \dots, 2,b_t \langle 1 \rangle \rp\]
or
\[CS(r)=\lp 2, b_t \langle 1 \rangle, \dots, 2, b_2 \langle 1 \rangle, 2, b_1 \langle 1 \rangle,
2, b_t \langle 1 \rangle, \dots, 2, b_2 \langle 1 \rangle, 2,b_1 \langle 1 \rangle \rp,\]
where $b_i=a_i-2$ for every $i$, and the symbol ``$b_i \langle 1 \rangle$'' represents $b_i$ successive $1$'s.
{\rm (}Here if $b_i=0$ for some $i$, then $b_i \langle 1 \rangle$ represents the empty subsequence.{\rm )}
\end{enumerate}
\end{proposition}

\begin{proof}
(1) Let $m_1\ge 2$. Write $\tilde{r}=q/p$, where $p$ and $q$ are relatively prime positive integers.
By Lemma ~\ref{lem:inductive_step}, $r=\tilde{r}/(1+\tilde{r})=q/(p+q)$.
It then follows from Lemma ~\ref{lem:half-rotation}
that both the sequences $S(r)$ and $S(\tilde{r})$, and hence both the cyclic sequences $CS(r)$ and $CS(\tilde{r})$,
have the same length $2q$.
Recall from \cite[Lemma ~4.8]{lee_sakuma} that
if $s_j(r)$ denotes the $j$-th term of the sequence $S(r)$, then
$s_j(r)=\lfloor j(1/r)\rfloor_*-\lfloor (j-1)(1/r)\rfloor_*$,
where $\lfloor x \rfloor_*$ is the greatest integer smaller than $x$.
Since $r=\tilde{r}/(1+\tilde{r})=1/(1/\tilde{r}+1)$, we have
\begin{align*}
s_j(r)
&=\lfloor j(1/r)\rfloor_*-\lfloor (j-1)(1/r)\rfloor_*\\
&=\lfloor j(1/\tilde{r}+1)\rfloor_* -
\lfloor (j-1)(1/\tilde{r}+1)\rfloor_*\\
&=(\lfloor j(1 / \tilde{r}) \rfloor_*+j) -
(\lfloor (j-1)(1 / \tilde{r})\rfloor_*+(j-1))\\
&=1+\lfloor j(1 / \tilde{r}) \rfloor_* -
\lfloor (j-1)(1 / \tilde{r})\rfloor_*\\
&=1+s_j({\tilde{r}}),
\end{align*}
where $s_j(\tilde{r})$ denotes the $j$-th term of the sequence $S(\tilde{r})$,
and hence the assertion follows.

(2) Let $m_1=1$. Then $\tilde{r}=[m_2+1, m_3, \dots, m_k]$ and $r=1-\tilde{r}$ by Lemma ~\ref{lem:inductive_step}.
Since $m_2+1 \ge 2$, (1) implies that every term of $CS(\tilde{r})$ is at least $2$,
that is, every $a_i$ is at least $2$.

To prove the remaining assertion,
let $f_1$ be the reflection of $(B^3,t(\infty))$
in a ``horizontal'' disk bounded by $\alpha_0$,
and let $f_2$ be the half Dehn twist of $(B^3,t(\infty))$
along the ``vertical'' disk bounded by $\alpha_{\infty}$.
Then the automorphisms $(f_i)_*$ of
$\pi_1(B^3-t(\infty))=F(a,b)$
induced by $f_i$ are given by
\[
(f_1)_*(a,b)=(a,b) \quad (f_2)_*(a,b)=(a,b^{-1})
\]
Let $f$ be the composition $f_2f_1$.
Then by the above observation, we have $f_*(a,b)=(a,b^{-1})$.
On the other hand, $f$ maps $\alpha_r$
to $f_2(f_1(\alpha_r))=f_2(\alpha_{-r})=\alpha_{1-r}=\alpha_{\tilde r}$.
Thus $f_*$ sends the cyclic word $(u_r)$ to the cyclic word
$(u_{\tilde r})$ or $(u_{\tilde r}^{-1})$.
Since $f_*^2=1$, this implies that
$f_*$ sends the cyclic word $(u_{\tilde r})$ to the cyclic word
$(u_r)$ or $(u_r^{-1})$.
Thus the cyclic word
$(u_r)$ or $(u_r^{-1})$ is obtained from
$(u_{\tilde r})$ by replacing $b$ with $b^{-1}$.
In this process, a subword, $w$, of $(u_{\tilde{r}})$ with $S(w)=(1, a_i, 1)$,
say, $w=b^{-1}(abab \cdots ab)a^{-1}$ or $b^{-1}(abab \cdots a)b^{-1}$
according to whether $a_i$ is even or odd,
is transformed to a subword $w'=b(ab^{-1}ab^{-1} \cdots ab^{-1})a^{-1}$ or $b(ab^{-1}ab^{-1} \cdots a)b$,
respectively, of $(u_r^{\pm 1})$ with $S(w')=(2, (a_i-2)\langle 1 \rangle, 2)$.
Since the cyclic sequence $CS(u_r^{-1})$ is the reverse of the cyclic sequence $CS(u_r)=CS(r)$,
the assertion now follows.
\end{proof}

Throughout the remainder of this paper, we assume the following well-ordering $\preceq$.

\begin{definition}
\label{def:well-ordering}
{\rm Let $\mathfrak{A}$ be the set of all rational numbers greater than $0$ and less than or equal to $1$.
We define a well-ordering $\preceq$ on $\mathfrak{A}$ by $r_1 \preceq r_2$
if and only if one of the following conditions holds,
where $r_1=[l_1,l_2, \dots,l_h]$ and $r_2=[n_1,n_2, \dots,n_t]$.
\begin{enumerate}[\indent \rm (i)]
\item
$h<t$.
\item
$h=t$ and there is a positive integer $j\le h=t$ such that
$l_i=n_i$ for every $i<j$ and $l_j\le n_j$.
\end{enumerate}
}
\end{definition}
It should be noted that a rational number $\tilde{r}$
defined in Lemma ~\ref{lem:inductive_step} is
a predecessor of $r=[m_1,m_2, \dots,m_k]$ with respect to $\preceq$.

Now we are able to give a new proof to the following lemma
whose statement is a part of \cite[Proposition ~4.3]{lee_sakuma}.
Note that the remaining part of \cite[Proposition ~4.3]{lee_sakuma}
is not necessary in the present paper.

\begin{lemma}
\label{lem:properties}
For a rational number $r=[m_1,m_2, \dots,m_k]$, we have the following.
\begin{enumerate}[\indent \rm (1)]
\item Suppose $k=1$, i.e., $r=1/m_1$.
Then $CS(r)= \lp m_1,m_1 \rp$.

\item Suppose $k\ge 2$. Then $CS(r)$ consists of $m_1$ and $m_1+1$.
\end{enumerate}
\end{lemma}

\begin{proof}
We prove (1) and (2) together by transfinite induction with respect to the well-ordering $\preceq$
defined in Definition ~\ref{def:well-ordering}.
The base step is the case when $r=[1]$. In this case, $u_r=ab^{-1}$, and so $CS(r)=\lp 1, 1 \rp$,
as desired. To prove the inductive step, we consider two cases separately.

\medskip
\noindent {\bf Case 1.} {\it $m_1 \ge 2$}.
\medskip

In this case, put $\tilde{r}=[m_1-1, m_2, \dots, m_k]$ as in Lemma ~\ref{lem:inductive_step}.
Then clearly $\tilde{r} \prec r$.
By the inductive hypothesis,
$CS(\tilde{r})=\lp m_1-1, m_1-1 \rp$ provided $k=1$, and
$CS(\tilde{r})$ consists of $m_1-1$ and $m_1$ provided $k \ge 2$.
So by Proposition ~\ref{prop:CS-sequence}(1),
$CS(r)=\lp m_1, m_1 \rp$ provided $k=1$, and
$CS(r)$ consists of $m_1$ and $m_1+1$ provided $k \ge 2$, as desired.

\medskip
\noindent {\bf Case 2.} {\it $m_1=1$}.
\medskip

In this case, it immediately follows from
Proposition ~\ref{prop:CS-sequence}(2)
that $CS(r)$ consists of $1=m_1$ and $2=m_1+1$, as desired.
\end{proof}

We also give a new proof to the following proposition
whose statement is precisely the same as \cite[Proposition ~4.5]{lee_sakuma}.

\begin{proposition}
\label{prop:decomposition}
For $r=[m_1, m_2, \dots, m_k]$,
the cyclic sequence $CS(r)$ has a decomposition $\lp S_1, S_2, S_1, S_2 \rp$
which satisfies the following.
\begin{enumerate}[\indent \rm (1)]
\item Each $S_i$ is symmetric, that is, the sequence obtained from $S_i$ by reversing
the order is equal to $S_i$. {\rm ({\it Here, $S_1$ is empty if $k=1$.})}
\item Each $S_i$ occurs only twice on the cyclic sequence $CS(r)$.
\item The subsequence $S_1$ begins and ends with $m_1+1$.
\item The subsequence $S_2$ begins and ends with $m_1$.
\end{enumerate}
\end{proposition}

\begin{proof}
The proof proceeds by transfinite induction with respect to the well-ordering $\preceq$
defined in Definition ~\ref{def:well-ordering}.
We take the case when $r=[m_1]$ as the base step.
In this case, $CS(r)=\lp m_1, m_1 \rp$ by Lemma ~\ref{lem:properties}(1).
Putting $S_1=\emptyset$ and $S_2=(m_1)$, the assertion clearly holds.
To prove the inductive step, we consider two cases separately.

\medskip
\noindent {\bf Case 1.} {\it $m_1 \ge 2$ and $k \ge 2$}.
\medskip

Put $\tilde{r}=[m_1-1, m_2, \dots, m_k]$ as in Lemma ~\ref{lem:inductive_step}.
Then clearly $\tilde{r} \prec r$.
By the inductive hypothesis,
$CS(\tilde{r})=\lp \tilde{S}_1, \tilde{S}_2, \tilde{S}_1, \tilde{S}_2 \rp$,
where $\tilde{S}_1$ and $\tilde{S}_2$ are symmetric subsequences of $CS(\tilde{r})$
such that each $\tilde{S}_i$ occurs only twice in $CS(\tilde{r})$,
$\tilde{S}_1$ begins and ends with $m_1$ (provided that $\tilde{S}_1$ is nonempty),
and such that $\tilde{S}_2$ begins and ends with $m_1-1$.
Write
\[
\tilde{S}_1=(a_1, \dots, a_{t_1}) \quad \text{\rm and} \quad
\tilde{S}_2=(a_{t_1+1}, \dots, a_{t_2}),
\]
and then take
\[
S_1=(a_1+1, \dots, a_{t_1}+1) \quad \text{\rm and} \quad
S_2=(a_{t_1+1}+1, \dots, a_{t_2}+1).
\]
Clearly $S_1$ begins and ends with $m_1+1$, and $S_2$ begins and ends with $m_1$.
Also since $\tilde{S}_1$ and $\tilde{S}_2$ are symmetric by the inductive hypothesis,
$S_1$ and $S_2$ are also symmetric.
Moreover, by Proposition ~\ref{prop:CS-sequence}(1), we have $CS(r)=\lp S_1, S_2, S_1, S_2 \rp$.
It remains to show that each $S_i$ occurs only twice in $CS(r)$.
If $S_1$ occurred more than twice in $\lp S_1, S_2, S_1, S_2 \rp$,
$\tilde{S}_1$ also would occur more than twice in $\lp \tilde{S}_1, \tilde{S}_2, \tilde{S}_1, \tilde{S}_2 \rp$,
a contradiction. Similarly, $S_2$ also occurs only twice in $CS(r)$.

\medskip
\noindent {\bf Case 2.} {\it $m_1=1$ and $k \ge 2$}.
\medskip

Put $\tilde{r}=[m_2+1, m_3, \dots, m_k]$ as in Lemma ~\ref{lem:inductive_step}.
Then clearly $\tilde{r} \prec r$.
By the inductive hypothesis,
$CS(\tilde{r})=\lp \tilde{S}_1, \tilde{S}_2, \tilde{S}_1, \tilde{S}_2 \rp$,
where $\tilde{S}_1$ and $\tilde{S}_2$ are symmetric subsequences of $CS(\tilde{r})$
such that each $\tilde{S}_i$ occurs only twice in $CS(\tilde{r})$,
$\tilde{S}_1$ begins and ends with $m_2+2$ (provided that $\tilde{S}_1$ is nonempty),
and such that $\tilde{S}_2$ begins and ends with $m_2+1$.
If $k=2$, then $r=[1,m_2]$ with $m_2 \ge 2$ and $\tilde{r}=[m_2+1]$;
so $CS(\tilde{r})=\lp m_2+1, m_2+1 \rp$ by Lemma ~\ref{lem:properties}(1).
Then take
\[
S_1=(2) \quad \text{\rm and} \quad S_2=((m_2-1) \langle 1 \rangle).
\]
On the other hand, if $k \ge 3$, then write
\[
\tilde{S}_1=(a_1, \dots, a_{t_1}) \quad \text{\rm and} \quad \tilde{S}_2=(a_{t_1+1}, \dots, a_{t_2}).
\]
Here $a_1=a_{t_1}=m_2+2 \ge 3$ and $a_{t_1+1}=a_{t_2}=m_2+1 \ge 2$.
Now take
\[
S_1=(2, b_{t_1+1}\langle 1 \rangle, 2, \dots, 2, b_{t_2}\langle 1 \rangle, 2)
\quad \text{\rm and} \quad
S_2=(b_1 \langle 1 \rangle, 2, \dots, 2, b_{t_1} \langle 1 \rangle),
\]
where $b_i=a_i-2$ for every $i$. In either case, we see that
$S_1$ begins and ends with $2=m_1+1$, $S_2$ begins and ends with $1=m_1$,
and that $S_1$ and $S_2$ are symmetric
because $\tilde{S}_1$ and $\tilde{S}_2$ are symmetric by the inductive hypothesis.
Moreover, Proposition ~\ref{prop:CS-sequence}(2) implies that
either $CS(r)=\lp S_1, S_2, S_1, S_2 \rp$ or
$CS(r)=\lp \overleftarrow{S_1}, \overleftarrow{S_2}, \overleftarrow{S_1}, \overleftarrow{S_2} \rp$,
where the symbol ``$\overleftarrow{S_i}$'' denotes the reverse of $S_i$.
But since $S_1$ and $S_2$ are symmetric, we actually have
$CS(r)=\lp S_1, S_2, S_1, S_2 \rp$ in either case.
It remains to show that each $S_i$ occurs only twice in $CS(r)$.
If $S_1$ occurred more than twice in $\lp S_1, S_2, S_1, S_2 \rp$,
$\tilde{S}_2$ also would occur more than twice in $\lp \tilde{S}_1, \tilde{S}_2, \tilde{S}_1, \tilde{S}_2 \rp$,
a contradiction.
For the assertion for $S_2$,
note that $S_2$ begins and ends with $m_2$ successive $1$'s,
and that the maximum number of consecutive occurrences
of $1$ in $\lp S_1, S_2, S_1, S_2 \rp$ is $m_2$.
So if $S_2$ occurred more than twice in $\lp S_1, S_2, S_1, S_2 \rp$,
$\tilde{S}_1$ also would occur more than twice in $\lp \tilde{S}_1, \tilde{S}_2, \tilde{S}_1, \tilde{S}_2 \rp$,
a contradiction.
\end{proof}

In order to prove Theorem ~\ref{thm:main_result}, we keep the idea of applying small cancellation theory
as in \cite[Sections ~5 and 6]{lee_sakuma}. Briefly speaking, we adopt \cite[Section ~5]{lee_sakuma} as it is
to show that the upper presentation $G(K(r))=\langle a, b \svert u_r \rangle$ with $0<r<1$
satisfies the small cancellation conditions $C(4)$ and $T(4)$.
And then we investigate properties of van Kampen's diagrams over the presentation $G(K(r))=\langle a, b \svert u_r \rangle$
with boundary label being cyclically alternating as in \cite[Section ~6]{lee_sakuma}.
Sections ~5 and 6 in \cite{lee_sakuma} are indeed
irrelevant to the modification that we are performing in the present paper.
Due to van Kampen's Lemma which is a classical result in combinatorial group theory (see \cite{lyndon_schupp}),
we obtain the fact that if a cyclically alternating word $w$ equals the identity
in $G(K(r))$, then its cyclic word $(w)$ contains a subword $z$ of $(u_r^{\pm 1})$
such that the $S$-sequence of $z$ is $(S_1, S_2, \ell)$ or $(\ell, S_2, S_1)$
for some positive integer $\ell$, where $CS(r)=\lp S_1, S_2, S_1, S_2 \rp$
is as in Proposition ~\ref{prop:decomposition}. In particular, we obtain the following.

\begin{corollary}{\cite[Corollary ~6.4]{lee_sakuma}}
\label{cor:identity}
Let $r=[m_1, m_2, \dots, m_k]$ with $0<r<1$.
For a rational number $s$ with $0<s\le 1$, if $\alpha_s$ is null-homotopic in $S^3-K(r)$,
then the following hold.
\begin{enumerate}[\indent \rm (1)]
\item If $k=1$, namely $r=[m_1]$, then $CS(s)$ contains a term bigger than or equal to $m_1$.

\item If $k\ge 2$, then $CS(s)$ contains $(S_1, S_2)$ or $(S_2, S_1)$ as a subsequence, where
$CS(r)=\lp S_1, S_2, S_1, S_2 \rp$ is as in Proposition ~\ref{prop:decomposition}.
\end{enumerate}
\end{corollary}

\begin{remark}
{\rm
In \cite[Corollary ~6.4]{lee_sakuma}, it is mistakenly stated that
if $\alpha_s$ is null-homotopic in $S^3-K(r)$,
then $CS(s)$ contains $(S_1, S_2)$ or $(S_2, S_1)$ as a subsequence, regardless of $k \ge 1$.
It should be noted that if $k=1$ and every term of $CS(s)$ is bigger than $m_1$,
then $CS(s)$ does not contain $(S_1, S_2)$ or $(S_2, S_1)$ as a subsequence,
because, in this case, $S_1$ is empty and $S_2=(m_1)$, i.e., $(S_1, S_2)=(m_1)=(S_2, S_1)$.
}
\end{remark}

\section{New proof of Theorem ~\ref{thm:main_result}}
\label{sec:new_proof_theorem}

In this section, we prove the only if part of Theorem ~\ref{thm:main_result},
that is, we prove that for any $s\in I_1\cup I_2$,
$\alpha_s$ is not null-homotopic in $S^3-K(r)$.
The if part is \cite[Corollary ~4.7]{Ohtsuki-Riley-Sakuma}.

The following lemma which plays an important role in the proof of
Theorem ~\ref{thm:main_result} has the same statement as \cite[Lemma ~7.3]{lee_sakuma},
but is re-proved by transfinite induction.

\begin{lemma}
\label{lem:connection}
Let $r=[m_1, m_2, \dots, m_k]$ with $0<r<1$,
and let $CS(r)=\lp S_1, S_2, S_1, S_2 \rp$ be as in Proposition ~\ref{prop:decomposition}.
Suppose that a rational number $s$ with $0<s \le 1$
has a continued fraction expansion $s=[l_1, \dots, l_t]$,
where $t \ge 1$, $(l_1, \dots, l_t) \in (\mathbb{Z}_+)^t$, and $l_t \ge 2$ unless $t=1$.
Suppose also that $CS(s)$ satisfies the following condition:
\begin{enumerate}[\indent \indent \rm (i)]
\item If $k=1$, namely $r=[m_1]$, then $CS(s)$ contains a term bigger than or equal to $m_1$.

\item If $k\ge 2$, then $CS(s)$ contains $(S_1, S_2)$ or $(S_2, S_1)$ as a subsequence.
\end{enumerate}
Then the following hold.
\begin{enumerate} [\indent \rm (1)]
\item $t \ge k$.

\item $l_i=m_i$ for each $i=1, \dots, k-1$.

\item Either $l_k \ge m_k$ or both $l_k=m_k-1$ and $t>k$.
\end{enumerate}
\end{lemma}

\begin{proof}
The proof proceeds by transfinite induction with respect to the well-ordering $\preceq$
defined in Definition ~\ref{def:well-ordering}.
We take the case when $r=[m_1]$ as the base. By hypothesis (i),
$CS(s)$ contains a term bigger than or equal to $m_1$.
Then Lemma ~\ref{lem:properties} implies that
either $l_1 \ge m_1$ or both $l_1=m_1-1$ and $t \ge 2$,
so that the assertion clearly holds.
Now we prove the inductive step.
Let $\tilde{r}$ be defined as in Lemma ~\ref{lem:inductive_step}.
Then clearly $\tilde{r} \prec r$.

\medskip
\noindent {\bf Case 1.} {\it $m_1 \ge 2$ and $k \ge 2$.}
\medskip

In this case, $\tilde{r}=[m_1-1, m_2, \dots, m_k]$.
By Proposition ~\ref{prop:decomposition},
$S_1$ begins and ends with $m_1+1$, and $S_2$ begins and ends with $m_1$.
Hence if $CS(s)$ contains $(S_1, S_2)$ or $(S_2, S_1)$ as a subsequence,
then $CS(s)$ contains both a term $m_1$ and a term $m_1+1$.
By Lemma ~\ref{lem:properties}, the only possibility is that $l_1=m_1$ and $t \ge 2$.
Now let $\tilde{s}=[l_1-1, \dots, l_t]$.
Then we see from Proposition ~\ref{prop:CS-sequence}(1) that
$CS(\tilde{s})$ contains $(\tilde{S}_1, \tilde{S}_2)$ or $(\tilde{S}_2, \tilde{S}_1)$ as a subsequence,
where $CS(\tilde{r})=\lp \tilde{S}_1, \tilde{S}_2, \tilde{S}_1, \tilde{S}_2 \rp$.
By the inductive hypothesis, we have $t \ge k$, $l_i=m_i$ for each $i=1, \dots, k-1$, and
either $l_k \ge m_k$ or both $l_k=m_k-1$ and $t>k$, which proves the assertion.

\medskip
\noindent {\bf Case 2.} {\it $m_1=1$ and $k\ge 2$.}
\medskip

In this case, $\tilde{r}=[m_2+1,m_3, \dots, m_k]$.
Arguing as in Case ~1, $CS(s)$ contains both a term $m_1=1$ and a term $m_1+1=2$.
By Lemma ~\ref{lem:properties}, the only possibility is that $l_1=m_1=1$ and $t \ge 2$.
Now let $\tilde{s}=[l_2+1, \dots, l_t]$.
Then we see from Proposition ~\ref{prop:CS-sequence}(2) that
$CS(\tilde{s})$ contains a term greater than or equal to $m_2+1$ provided $k=2$
and that $CS(\tilde{s})$ contains $(\tilde{S}_1, \tilde{S}_2)$ or $(\tilde{S}_2, \tilde{S}_1)$ as a subsequence
provided $k \ge 3$, where $CS(\tilde{r})=\lp \tilde{S}_1, \tilde{S}_2, \tilde{S}_1, \tilde{S}_2 \rp$.
By the inductive hypothesis, we have $t \ge k$, $l_i=m_i$ for each $i=2, \dots, k-1$, and
either $l_k \ge m_k$ or both $l_k=m_k-1$ and $t>k$. This together with $l_1=m_1$ proves the assertion.
\end{proof}

\begin{remark}\rm
\label{rem:connection}
We can easily see that the a rational number $s$ with $0<s\le 1$
satisfies the conclusion of Lemma ~\ref{lem:connection}
if and only if $s$ lies in the open interval
$(r_1,r_2)=(0,1]-(I_1\cup I_2)$,
where $r_1$ and $r_2$ are rational numbers
such that $I_1=[0,r_1]$ and $I_2=[r_2,1]$,
introduced in Section ~\ref{sec:main_statement}.
\end{remark}

We are now in a position to prove the only if part of Theorem ~\ref{thm:main_result}.

\begin{proof} [Proof of the only if part of Theorem ~\ref{thm:main_result}]
Since the exceptional cases $r=\infty$ and $r=1$ can be treated in the same way as in \cite[Section ~7]{lee_sakuma},
we assume $0<r<1$.
Consider a $2$-bridge link $K(r)$, and pick a rational number $s$ from $I_1\cup I_2$.
Suppose on the contrary that $\alpha_s$ is null-homotopic in $S^3-K(r)$, namely $u_s=1$ in $G(K(r))$.
If $0<s\le 1$, then by Corollary ~\ref{cor:identity},
$CS(s)$ contains a term greater than or equal to $m_1$ provided $r=[m_1]$
or otherwise $CS(s)$ contains $(S_1, S_2)$ or $(S_2, S_1)$ as a subsequence,
where $CS(r)=\lp S_1, S_2, S_1, S_2 \rp$ as in Proposition ~\ref{prop:decomposition}.
But then by Lemma ~\ref{lem:connection} together with Remark ~\ref{rem:connection},
we have $s \notin I_1\cup I_2$, a contradiction.
So the only possibility is $s=0$.
Then, as mentioned at the end of Section ~\ref{sec:new_proof_small_cancellation}
(also see \cite[Theorem ~6.3]{lee_sakuma}),
$u_s$ must contain a subword $z$ of $(u_r^{\pm 1})$
such that the $S$-sequence of $z$ is
$(S_1, S_2, \ell)$ or $(\ell, S_2, S_1)$
for some positive integer $\ell$.
Note that the length of such a subword $z$ is strictly greater than
$p$, half the length of $(u_r^{\pm 1})$, where $r=q/p$.
Since $0<r<1$, we have $p\ge 2$.
So, the word $u_0=ab$ cannot contain such a subword, a contradiction.
This completes the proof of the only if part of Theorem ~\ref{thm:main_result}.
\end{proof}

\section*{Acknowledgement}
The authors are heartily grateful to Makoto Sakuma and an anonymous referee
for their valuable comments and suggestions
which led to an improvement of this paper.

\bibstyle{plain}

\end{document}